\documentclass[11pt,a4paper,onesided]{article}

\usepackage{ amssymb, amsmath, amsthm, graphics, xspace, enumerate}
\usepackage{graphicx}
\usepackage[a4paper,colorlinks=true,citecolor=blue,urlcolor=blue,linkcolor=blue,bookmarksopen=true,unicode=true,pdffitwindow=true]{hyperref}
 \hypersetup{pdfauthor={Nikola Sandri\'{c}}}
\hypersetup{pdftitle={On recurrence and transience of two-dimensional L\'evy and L\'evy-type processes}}

\newtheorem{tm}{tm}[section]
\newtheorem{theorem}[tm]{Theorem}

\newtheorem{corollary}[tm]{Corollary}
\newtheorem{proposition}[tm]{Proposition}

\newtheorem{example}[tm]{Example}

\newcommand{\process}[1]{\{#1_t\}_{t\geq0}}

\newcommand {\R} {\ensuremath{\mathbb{R}}}

\newcommand {\N} {\ensuremath{\mathbb{N}}}
\newcommand {\CC} {\ensuremath{\mathbb{C}}}
\numberwithin{equation}{section}
\def\be{\begin{equation}}
\def\ee{\end{equation}}
\begin{document}

\title{On Recurrence and Transience of Two-Dimensional L\'evy and L\'evy-Type Processes}
 \author{Nikola Sandri\'{c}\\
Institut f\"ur Mathematische Stochastik,\\ Fachrichtung Mathematik, Technische Universit\"at Dresden, Dresden, Germany\\
and\\
Department of Mathematics\\
         Faculty of Civil Engineering, University of Zagreb, Zagreb,
         Croatia \\
        Email: nsandric@grad.hr }

 \maketitle
\begin{center}
{
\medskip

} \end{center}

\begin{abstract}
In this paper, we study recurrence and transience
   of L\'evy-type processes, that is,   Feller processes associated with
pseudo-differential operators. Since  the recurrence property of L\'evy-type processes in dimensions greater than two is vacuous and  the recurrence and transience of one-dimensional L\'evy-type processes have been very well investigated, in this paper we are focused on the two-dimensional case only.
In particular, we study perturbations of two-dimensional L\'evy-type processes
 which do not affect their recurrence and transience
properties,  we derive sufficient conditions for their recurrence
and transience in terms of the corresponding L\'evy measure and we provide
 some comparison conditions for the recurrence and transience
also in terms of the L\'evy measures.

\end{abstract}
\smallskip


%
%
%
%


\section{Introduction}\label{s1}

Let $(\{L_t\}_{t\geq0},\{\mathbb{P}^{x}\}_{x\in\R^{d}})$ be a $d$-dimensional, $d\geq1$, L\'evy
process.  The process $\{L_t\}_{t\geq0}$ is said
to be \emph{recurrent} if
$$\int_0^{\infty}\mathbb{P}^{x}\left(L_t\in B_r(x)\right)dt=\infty\quad\textrm{for all}\ x\in\R^{d}\ \textrm{and all}\ r>0,$$
and
 \emph{transient} if
$$\int_0^{\infty}\mathbb{P}^{x}\left(L_t\in B_r(x)\right)dt<\infty\quad\textrm{for all}\ x\in\R^{d}\ \textrm{and all}\ r>0.$$
Here, $B_r(x)$ denotes the open ball of radius $r>0$ around $x\in\R^{d}$.
It is well known that every L\'{e}vy process is either recurrent or transient (see \cite[Theorem 35.3]{sato-book}).
However, the above definitions (characterizations) of the recurrence and transience properties are not  practical in most cases.
Due to the stationarity and independence of the increments, every L\'evy process $\{L_t\}_{t\geq0}$ can be uniquely and completely characterized through  the characteristic function of a single random variable $L_t$, $t>0$, that is, by the famous \emph{L\'evy-Khintchine
formula} we have
$$\mathbb{E}^{x}[\exp\{i\langle\xi,L_t-x\rangle\}]=\exp\{-tq(\xi)\},\quad  t\geq0,\ x\in\R^{d},$$
where the function $q:\R^{d}\longrightarrow\CC$ is called the \emph{characteristic exponent} of the process $\{L_t\}_{t\geq0}$ and it enjoys the following L\'evy-Khintchine representation
$$q(\xi)=i\langle\xi,b\rangle+\frac{1}{2}\langle\xi,C\xi\rangle+\int_{\R^{d}}\left(1-\exp\{i\langle\xi,y\rangle\}+i\langle\xi,y\rangle1_{B_1(0)}(y)\right)\nu(dy),\quad \xi\in\R^{d}.$$
Here,  $b\in\R^{d}$, $C$ is a symmetric non-negative definite $d\times d$ matrix and $\nu(dy)$ is a $\sigma$-finite Borel measure on $\R^{d}$, the so-called \emph{L\'evy measure}, satisfying $\int_{\R^{d}}\min\{1,|y|^{2}\}\nu(y)<\infty$.
Now, by having this nice analytical description (characterization) of L\'{e}vy processes,
 a more
operable
 characterization of the recurrence and transience properties  has been given by the
 well-known
\emph{Chung-Fuchs criterion}. A L\'evy process $\{L_t\}_{t\geq0}$ is
recurrent if, and only if,
$$\int_{B_r(0)}{\rm Re}\left(\frac{1}{q(\xi)}\right)d\xi=\infty
\quad \textrm{for some (all)}\ r>0,$$ (see \cite[Corollary 37.6 and
Remark 37.7]{sato-book}). As one of the  direct consequences of this criterion we get that in  dimensions greater than two every L\'evy process is  transient (see \cite[Theorem 37.8]{sato-book}).

The notion of  recurrence and transience can  also be defined for a broader class of processes.
Let $(\process{M},\{\mathbb{P}^{x}\}_{x\in\R^{d}})$ be a  $d$-dimensional, $d\geq1$, Markov process with c\`adl\`ag sample paths. The
process $\process{M}$ is called
\begin{enumerate}
  \item [(i)] $\varphi$\emph{-irreducible} if there exists a $\sigma$-finite measure $\varphi(dy)$ on
$\mathcal{B}(\R^{d})$ such that whenever $\varphi(B)>0$ we have
$\int_0^{\infty}\mathbb{P}^{x}(M_t\in B)dt>0$ for all $x\in\R^{d}$.
  \item [(ii)] \emph{recurrent} if it is
                      $\varphi$-irreducible and if $\varphi(B)>0$ implies $\int_{0}^{\infty}\mathbb{P}^{x}(M_t\in B)dt=\infty$ for all
                      $x\in\R^{d}$.
 \item [(iii)] \emph{transient} if it is $\varphi$-irreducible
                       and if there exists a countable
                      covering of $\R^{d}$ with  sets
$\{B_j\}_{j\in\N}\subseteq\mathcal{B}(\R^{d})$, such that for each
$j\in\N$ there is a finite constant $c_j\geq0$ such that
$\int_0^{\infty}\mathbb{P}^{x}(M_t\in B_j)dt\leq c_j$ holds for all
$x\in\R^{d}$.
\end{enumerate}
Recall that, as in the L\'evy process case, every $\varphi$-irreducible Markov
process is either recurrent or transient  (see \cite[Theorem
2.3]{tweedie-mproc}). Now, let $(\process{F},\{\mathbb{P}^{x}\}_{x\in\R^{d}})$ be a $d$-dimensional, $d\geq 1$, conservative (that is, $\mathbb{P}^{x}(F_t\in\R^{d})=1$ for all $t\geq0$ and all $x\in\R^{d}$) Feller process
with infinitesimal generator $(\mathcal{A},\mathcal{D}_{\mathcal{A}})$. If  the set of smooth functions
with compact support $C_c^{\infty}(\R^{d})$ is contained in $\mathcal{D}_{\mathcal{A}}$, then the operator $\mathcal{A}|_{C_c^{\infty}(\R^{d})}$ has the following representation $$\mathcal{A}|_{C_c^{\infty}(\R^{d})}f(x) = -\int_{\R^{d}}q(x,\xi)\exp\{i\langle \xi,x\rangle\}
\hat{f}(\xi) d\xi,$$
 where $\hat{f}(\xi)$ denotes the Fourier transform of $f(x)$ and the function $q:\R^{d}\times\R^{d}\longrightarrow\CC$ is called the \emph{symbol} of the operator $\mathcal{A}|_{C_c^{\infty}(\R^{d})}$  (process $\{F_t\}_{t\geq0}$) and, for each fixed $x\in\R^{d}$, it is the characteristic exponent of some L\'evy process. In particular, it enjoys the ($x$-dependent)  L\'evy-Khintchine representation
$$q(x,\xi) =i\langle \xi,b(x)\rangle + \frac{1}{2}\langle\xi,C(x)\xi\rangle +
\int_{\R^{d}}\left(1-\exp\{i\langle\xi,y\rangle\}+i\langle\xi,y\rangle1_{B_1(0)}(y)\right)\nu(x,dy).$$ Accordingly, a Feller process satisfying the above properties is called  a \emph{L\'evy-type process}  (see Section \ref{s2} for details). Observe that this class of processes contains the class of L\'evy processes.

In
the sequel, we consider  only the so-called \emph{open-set irreducible}
L\'evy-type processes, that is,
L\'evy-type processes whose  irreducibility measure is fully supported.
An example of such measure is the Lebesgue measure, which we denote by $\lambda(dy)$. Clearly,  a L\'evy-type process $\process{F}$ will be
$\lambda$-irreducible if $\mathbb{P}^{x}(F_t\in B)>0$ for all
 $x\in\R^{d}$ and all $t>0$ whenever $\lambda(B)>0.$ In particular, the
process $\process{F}$ will be $\lambda$-irreducible if the
transition kernel $\mathbb{P}^{x}(F_t\in dy)$, $t>0$, $x\in\R^{d}$,  possesses a strictly positive transition density
function. Let us remark that the $\lambda$-irreducibility of L\'evy-type processes is a very well-studied topic in the literature. We refer the readers to  \cite{Sheu-1991} and \cite{Stramer-Tweedie-1997} for the case of elliptic diffusion processes, to \cite{Kolokoltsov-2000}  for the case of a class of pure jump L\'evy-type processes (the so-called stable-like processes),
to   \cite{Bass-Cranston-1986},  \cite{Ishikawa-2001}, \cite{Knopova-Kulik-2014}, \cite{Kulik-2007}, \cite{Kwoon-Lee-1999} and \cite{Picard-1996, Picard-Erratum-2010} for the case of a class of L\'evy-type processes obtained as a solution of certain jump-type stochastic differential equations and
\cite{Knopova-Schilling-2012}, \cite{Knopova-Schilling-2013} and \cite{Pang-Sandric-2015}    for the case of general L\'evy-type processes.

Now, if
 $\process{F}$ is an open-set irreducible L\'evy-type process with symbol $q(x,\xi)$,
then, under certain additional regularity conditions on the symbol  (see conditions \textbf{(C2)} and \textbf{(C3)} in Section \ref{s2}), in \cite{sandric-TAMS} and \cite{rene-wang-feller} it has been proven that $\process{F}$ is recurrent if
 $$\liminf_{\alpha\longrightarrow0}\int_{\R^{d}}\left(\int_{\frac{\ln 2}{4\sup_{x\in\R^{d}}|q(x,\xi)|}}^{\infty}e^{-\alpha t}\,{\rm Re}\,\mathbb{E}^{0}[\it{e}^{\it{i}\langle\xi, F_t\rangle}]dt\right)\frac{\sin^{2}\left(\frac{a\xi_1}{2}\right)}{
\xi_1^{2}}\cdots\frac{\sin^{2}\left(\frac{a\xi_d}{2}\right)}{
\xi_d^{2}}d\xi>-\infty$$ for all $a>0$ small enough (see \cite[Proposition 2.5]{sandric-transience} for further discussion on this condition)
 and
\begin{equation}\label{eq1.1}\int_{B_r(0)}\frac{d\xi}{\sup_{x\in\R^{d}}|q(x,\xi)|}=\infty\quad \textrm{for some}\ r>0,\end{equation} and  it is transient if $\sup_{x\in\R^{d}}|{\rm Im}\,q(x,\xi)|\leq c\inf_{x\in\R^{d}}{\rm Re}\,q(x,\xi)$ for some $0\leq c<1$ and all $\xi\in\R^{d}$ and\begin{equation}\label{eq1.2}\int_{B_r(0)}\frac{d\xi}{\inf_{x\in\R^{d}}{\rm Re}\,q(x,\xi)}<\infty\quad \textrm{for some}\ r>0.\end{equation}
Again, directly from the above Chung-Fuchs type conditions, in \cite[Theorem 2.8]{sandric-TAMS} it has been shown that  in  dimensions greater than two  every L\'evy-type process (satisfying the above mentioned conditions) is transient.

Due to this, the problem of   recurrence and transience of L\'evy and L\'evy-type processes reduces to the dimensions one and two. In the one-dimensional case this problem has been  extensively studied in \cite{sandric-TAMS} and \cite{sato-book}. The main   questions regarding the recurrence and transience addressed in these references consider a connection with Pruitt indices, problem of
perturbations of these processes, conditions for the recurrence and transience in terms of
the underlying L\'evy measure and  problem of comparison for
the recurrence and transience also in terms of the  L\'evy measures.

In this work, we are focused on the same questions as in the one-dimensional case. More precisely,
 in Theorem
\ref{tm3.2}, we prove that if $\process{F}$ and
$\process{\bar{F}}$ are two-dimensional L\'evy-type processes with radial (in the co-variable) symbols $q(x,\xi)$ and $\bar{q}(x,\xi)$ and   L\'evy measures
$\nu(x,dy)$ and $\bar{\nu}(x,dy)$, respectively,  then $\process{F}$  and $\process{\bar{F}}$ are recurrent or
transient at the same time if there is a rotation of the plane (orthogonal matrix) $R$ such that
$$\sup_{x\in\R^{2}}\int_{\R^{2}}|y|^{2}|\nu(x,dy)-\bar{\nu}(Rx,dy)|<\infty.$$ Here,
$|\mu(dy)|$
denotes the
total variation measure of the signed measure $\mu(dy)$.
Note that this result  automatically implies
 that the  recurrence and transience of L\'evy-type processes depend only on the nature of big jumps of the process.
Further, since in general it is not always easy to check the
Chung-Fuchs type conditions in \eqref{eq1.1} and \eqref{eq1.2},
 in
Theorem \ref{tm4} we give necessary and
sufficient conditions for the recurrence and transience in terms of
the L\'evy measure. More precisely, under the assumption that the corresponding symbol is radial in the co-variable and certain additional regularity conditions,  we prove that \eqref{eq1.1} is equivalent to
 $$\int_r^{\infty}\left(\rho\sup_{x\in\R^{2}}\int_0^{\rho}u\,\nu(x,B_u^{c}(0))du\right)^{-1}d\rho=\infty\quad\textrm{for some}\ r>0,$$
and
\eqref{eq1.2} is equivalent to  $$\int_r^{\infty}\left(\rho\inf_{x\in\R^{2}}\int_0^{\rho}u\,\nu(x,B_u^{c}(0))du\right)^{-1}d\rho<\infty\quad\textrm{for some}\ r>0.$$
Finally, in Theorem \ref{tm3.10}, we give some comparison conditions for the recurrence
and transience  by comparing  ``tails" of the L\'evy measures, that is, we prove that the recurrence of the process with the L\'evy measure with ``bigger tail" implies the recurrence of the one with ``smaller tail". Similarly, we prove that the transience of the process with the L\'evy measure with ``smaller tail" implies the transience of the process with ``bigger tail".

The remaining part of this paper is organized as follows. In Section \ref{s2}, we give some preliminaries on L\'evy and L\'evy-type processes and in Section \ref{s3} we discuss   perturbations of these processes. In Section \ref{s4}  we derive sufficient conditions for the recurrence and transience in terms of
the L\'evy measure and, finally, in Section \ref{s5} we give some comparison conditions for the recurrence and transience properties
also in terms of the L\'evy measures.

\section{Preliminaries on L\'evy and L\'evy-Type Processes}\label{s2}

 Let
$(\Omega,\mathcal{F},\{\mathbb{P}^{x}\}_{x\in\R^{d}},\process{\mathcal{F}},\process{\theta},\process{M})$, denoted by $\process{M}$
in the sequel, be a Markov process with  state space
$(\R^{d},\mathcal{B}(\R^{d}))$, where $d\geq1$ and
$\mathcal{B}(\R^{d})$ denotes the Borel $\sigma$-algebra on
$\R^{d}$. A family of linear operators $\process{P}$ on
$B_b(\R^{d})$ (the space of bounded and Borel measurable functions),
defined by $$P_tf(x):= \mathbb{E}^{x}[f(M_t)],\quad t\geq0,\
x\in\R^{d},\ f\in B_b(\R^{d}),$$ is associated with the process
$\process{M}$. Since $\process{M}$ is a Markov process, the family
$\process{P}$ forms a \emph{semigroup} of linear operators on the
Banach space $(B_b(\R^{d}),||\cdot||_\infty)$, that is, $P_s\circ
P_t=P_{s+t}$ and $P_0=I$ for all $s,t\geq0$. Here,
$||\cdot||_\infty$ denotes the supremum norm on the space
$B_b(\R^{d})$. Moreover, the semigroup $\process{P}$ is
\emph{contractive}, that is, $||P_tf||_{\infty}\leq||f||_{\infty}$
for all $t\geq0$ and all $f\in B_b(\R^{d})$, and \emph{positivity
preserving}, that is, $P_tf\geq 0$ for all $t\geq0$ and all $f\in
B_b(\R^{d})$ satisfying $f\geq0$. The \emph{infinitesimal generator}
$(\mathcal{A}_{b},\mathcal{D}_{\mathcal{A}_{b}})$ of the semigroup
$\process{P}$ (or of the process $\process{M}$) is a linear operator
$\mathcal{A}_{b}:\mathcal{D}_{\mathcal{A}_{b}}\longrightarrow B_b(\R^{d})$
defined by
$$\mathcal{A}_{b}f:=
  \lim_{t\longrightarrow0}\frac{P_tf-f}{t},\quad f\in\mathcal{D}_{\mathcal{A}_{b}}:=\left\{f\in B_b(\R^{d}):
\lim_{t\longrightarrow0}\frac{P_t f-f}{t} \ \textrm{exists in}\
||\cdot||_\infty\right\}.
$$ We call $(\mathcal{A}_{b},\mathcal{D}_{\mathcal{A}_{b}})$ the \emph{$B_b$-generator} for short.
A Markov process $\process{M}$ is said to be a \emph{Feller process}
if its corresponding  semigroup $\process{P}$ forms a \emph{Feller
semigroup}. This means that the family $\process{P}$ is a semigroup
of linear operators on the Banach space
$(C_\infty(\R^{d}),||\cdot||_{\infty})$  and it is \emph{strongly
continuous}, that is,
  $$\lim_{t\longrightarrow0}||P_tf-f||_{\infty}=0,\quad f\in
  C_\infty(\R^{d}).$$ Here, $C_\infty(\R^{d})$ denotes
the space of continuous functions vanishing at infinity.  Note that
every Feller semigroup $\process{P}$  can be uniquely extended to
$B_b(\R^{d})$ (see \cite[Section 3]{rene-conserv}). For notational
simplicity, we denote this extension again by $\process{P}$. Also,
let us remark that every Feller process possesses the strong Markov
property and has c\`adl\`ag sample paths (see  \cite[Theorems 3.4.19 and
3.5.14]{jacobIII}).   Further,
in the case of Feller processes, we call
$(\mathcal{A},\mathcal{D}_{\mathcal{A}}):=(\mathcal{A}_{b},\mathcal{D}_{\mathcal{A}_{b}}\cap
C_\infty(\R^{d}))$ the \emph{Feller generator} for short. Note
that, in this case, $\mathcal{D}_{\mathcal{A}}\subseteq
C_\infty(\R^{d})$ and
$\mathcal{A}(\mathcal{D}_{\mathcal{A}})\subseteq
C_\infty(\R^{d})$. If  the Feller
generator
$(\mathcal{A},\mathcal{D}_{\mathcal{A}})$ of a
Feller process $\process{M}$ satisfies:
    \begin{description}
      \item[(\textbf{C1})]
      $C_c^{\infty}(\R^{d})\subseteq\mathcal{D}_{\mathcal{A}}$,
    \end{description}
 then, according to \cite[Theorem 3.4]{courrege-symbol},
$\mathcal{A}|_{C_c^{\infty}(\R^{d})}$ is a \emph{pseudo-differential
operator}, that is, it can be written in the form
\begin{equation}\label{eq2.1}\mathcal{A}|_{C_c^{\infty}(\R^{d})}f(x) = -\int_{\R^{d}}q(x,\xi)\exp\{i\langle \xi,x\rangle\}
\hat{f}(\xi) d\xi,\end{equation}
 where $\hat{f}(\xi):=
(2\pi)^{-d} \int_{\R^{d}} e^{-i\langle\xi,x\rangle} f(x) dx$ denotes
the Fourier transform of the function $f(x)$.
 The function $q :
\R^{d}\times \R^{d}\longrightarrow \CC$ is called  the \emph{symbol}
of the pseudo-differential operator. It is measurable and locally
bounded in $(x,\xi)$ and continuous and negative definite as a
function of $\xi$. Hence, by \cite[Theorem 3.7.7]{jacobI}, the
function $\xi\longmapsto q(x,\xi)$ has for each $x\in\R^{d}$ the
 L\'{e}vy-Khintchine representation
 \begin{equation}\label{eq2.2}q(x,\xi) =a(x)-
i\langle \xi,b(x)\rangle + \frac{1}{2}\langle\xi,C(x)\xi\rangle +\int_{\R^{d}}\left(1-\exp\{i\langle\xi,y\rangle\}+i\langle\xi,y\rangle1_{B_1(0)}(y)\right)\nu(x,dy),\end{equation}
 where $a(x)$ is a nonnegative Borel measurable function, $b(x)$ is
an $\R^{d}$-valued Borel measurable function,
$C(x):=(c_{ij}(x))_{1\leq i,j\leq d}$ is a symmetric non-negative
definite $d\times d$ matrix-valued Borel measurable function
 and $\nu(x,dy)$ is a Borel kernel on $\R^{d}\times
\mathcal{B}(\R^{d})$, called the \emph{L\'evy measure}, satisfying
$$\nu(x,\{0\})=0\quad \textrm{and} \quad \int_{\R^{d}}\min\{1,
|y|^{2}\}\nu(x,dy)<\infty,\quad x\in\R^{d}.$$ The quadruple
$(a(x),b(x),c(x),\nu(x,dy))$ is called the \emph{L\'{e}vy quadruple}
of the pseudo-differential operator
$\mathcal{A}|_{C_c^{\infty}(\R^{d})}$ (or of the symbol $q(x,\xi)$).
Let us remark that the local boundedness of $q(x,\xi)$  implies that for every compact set $K\subseteq\R^{d}$ there exists a finite constant $c_K>0$ such that
\begin{equation}\label{eq2.3}\sup_{x\in K}|q(x,\xi)|\leq c_K(1+|\xi|^{2}),\quad \xi\in\R^{d},\end{equation} (see \cite[Lemma 3.6.22]{jacobI}).
Due to \cite[Lemma 2.1 and Remark 2.2]{rene-holder},  \eqref{eq2.3} is equivalent to the local boundedness of the L\'evy quadruple, that is, for every compact set $K\subseteq\R^{d}$ we have $$\sup_{x\in K}a(x)+\sup_{x\in K}|b(x)|+\sup_{x\in K}|c(x)|+\sup_{x\in K}\int_{\R^{d}}\min\{1,y^{2}\}\nu(x,dy)<\infty.$$ According to the same reference, the global boundedness of the L\'evy quadruple is equivalent to
 \begin{description}
      \item[(\textbf{C2})] $||q(\cdot,\xi)||_\infty\leq c(1+|\xi|^{2})$  for some
  finite $c>0$ and all $\xi\in\R^{d}$.
       \end{description}
In the sequel, we also assume the following condition on the symbol
$q(x,\xi)$:
\begin{description}
  \item[(\textbf{C3})] $q(x,0)=a(x)=0$ for all $x\in\R^{d}.$
\end{description}
Let us remark that, according to \cite[Theorem 2.34]{rene-bjorn-jian},
condition (\textbf{C3}) is closely related to the conservativeness property of $\process{M}$. Namely, $\process{M}$ is \emph{conservative}, that is, $\mathbb{P}^{x}(M_t\in\R^{d})=1$ for all $t\geq0$ and all
$x\in\R^{d}$, if $q(x,\xi)$ satisfies (\textbf{C3}) and \begin{equation}\label{eq2.4}\liminf_{k\longrightarrow\infty}\sup_{|y-x|\leq k}\sup_{|\xi|\leq1/2k}|q(y,\xi)|<\infty,\quad x\in\R^{d}.\end{equation}  Moreover, due to \cite[Theorem 5.2]{rene-conserv}, under \textbf{(C2)},  (\textbf{C3}) is equivalent to
 the conservativeness property of the process $\process{M}$. Further,  note that in the case when the symbol $q(x,\xi)$ does not depend
on the variable $x\in\R^{d}$, $\process{M}$ becomes a \emph{L\'evy
process}, that is, a stochastic process   with stationary and
independent increments and c\`adl\`ag sample paths.
 Moreover, every L\'evy process is uniquely
and completely characterized through its corresponding symbol (see \cite[Theorems 7.10 and
8.1]{sato-book}). According to this, it is not hard to check that every L\'evy process satisfies conditions
(\textbf{C1}), (\textbf{C2}) and (\textbf{C3}) (see \cite[Theorem 31.5]{sato-book}). Thus, the class of processes we consider in this paper contains a subclass of L\'evy
processes. Let us also remark here that, unlike in the case of L\'evy processes, it is not possible
to associate a Feller process to every  symbol (see \cite{rene-bjorn-jian} for  details).
 Throughout this paper, the symbol $\process{F}$ denotes a Feller
process satisfying conditions (\textbf{C1}), (\textbf{C2}) and (\textbf{C3}). Such a
process is called a \emph{L\'evy-type process}. Also, a L\'evy process is denoted by $\process{L}$.
If  $\nu(x,dy)=0$ for all $x\in\R^{d}$, according to \cite[Theorem 2.44]{rene-bjorn-jian}, $\process{F}$ becomes an \emph{elliptic diffusion process}. Note that this definition agrees
with the standard definition of elliptic diffusion processes (see \cite{Rogers-Williams-Book-I-2000}).
 For more
on L\'evy and L\'evy-type processes  we refer the readers to the monographs
\cite{sato-book} and \cite{rene-bjorn-jian}.

\section{Perturbations of L\'evy and L\'evy-Type Processes}\label{s3}
In this section,
 we study perturbations of two-dimensional L\'evy and L\'evy-type processes
 which do not affect their recurrence and transience properties.
We start  with the following proposition which we will need in the sequel.
\begin{proposition}\label{p3.1} Let $\process{F}$  be
a $d$-dimensional L\'evy-type process with Feller generator $(\mathcal{A},\mathcal{D}_{\mathcal{A}})$, symbol $q(x,\xi)$ and  L\'evy triplet $(b(x),C(x),\nu(x,dy))$, and let $M$ be a regular $d\times d$ matrix. Then, the process  $\process{MF}$  is also
a $d$-dimensional L\'evy-type process which is determined by   symbol and L\'evy triplet of the form \begin{align}\label{eq3.1}q_{M}(x,\xi)&=q(M^{-1}x,M^{T}\xi),\nonumber\\ b_{M}(x)&=Mb(M^{-1}x)+\int_{\R^{d}}My\left(1_{B_1(0)}(y)-1_{B_1(0)}(My)\right)\nu(x,dy),\nonumber\\
C_{M}(x)&=MC(M^{-1}x)M^{T},\nonumber\\
\nu_{M}(x,dy)&=\nu(M^{-1}x,M^{-1}dy),\end{align} respectively. Here, $M^{T}$ denotes the transpose matrix of the matrix $M$. Further, if  $\process{F}$ is open-set irreducible, then $\process{MF}$ is also open-set irreducible.
\end{proposition}
\begin{proof}
First, by a straightforward inspection, it is easy to see that $\process{MF}$ is a Feller process with respect to $\mathbb{P}_M^{x}(MF_t\in dy):=\mathbb{P}^{M^{-1}x}(F_t\in M^{-1}dy),$ $t\geq0$, $x\in\R^{d}$. Next, let us show that $\process{MF}$ satisfies \textbf{(C1)} and that its symbol and L\'evy triplet are given by the relations in \eqref{eq3.1}. Since,
$$\int_{\R^{d}}f(y)\mathbb{P}_M^{x}(MF_t\in dy)-f(x)=\int_{\R^{d}}f\circ M(y)\mathbb{P}^{M^{-1}x}(F_t\in dy)-f\circ M(M^{-1}x),\quad x\in\R^{d},\ f\in B_b(\R^{d}),$$
 we easily conclude that $C_c^{\infty}(\R^{d})\subseteq \mathcal{D}_{\mathcal{A}_{M}}$ and $\mathcal{A}_{M}f(x)=\mathcal{A}f\circ M(M^{-1}x)$ for $x\in\R^{d}$ and $f\in C_c^{\infty}(\R^{d})$,  where $(\mathcal{A}_{M},\mathcal{D}_{\mathcal{A}_{M}})$ denotes the Feller generator of $\process{MF}$. Now, according to \eqref{eq2.1}, for $x\in\R^{d}$ and $f\in C_c^{\infty}(\R^{d})$,
\begin{align*}
\mathcal{A}_{M}f(x)&=-\int_{\R^{d}}q_{M}(x,\xi)e^{i\langle\xi,x\rangle}\hat{f}(\xi)d\xi\\
&=-\int_{\R^{d}}q(M^{-1}x,\xi)e^{i\langle\xi,M^{-1}x\rangle}\widehat{f\circ M}(\xi)d\xi\\
&=-|\det M^{-1}|\int_{\R^{d}}q(M^{-1}x,\xi)e^{i\langle\xi,M^{-1}x\rangle}\hat{f}\left(\left(M^{-1}\right)^{T}\xi\right)d\xi\\
&=-\int_{\R^{d}}q(M^{-1}x,M^{T}\xi)e^{i\langle\xi,x\rangle}\hat{f}\left(\xi\right)d\xi,
\end{align*}
which together with \eqref{eq2.2} proves the assertion. Finally, it is straightforward to see that $\process{MF}$ satisfies the conditions in \textbf{(C2)} and \textbf{(C3)} and that it is open-set irreducible.
\end{proof}

Recall that a \emph{rotation of the plane} for an angle $\varphi\in[0,2\pi)$  is a linear mapping $R_\varphi:\R^{2}\longrightarrow\R^{2}$ represented by the  matrix $$R_{\varphi}=\left[
                                               \begin{array}{cc}
                                                 \cos \varphi & -\sin \varphi \\
                                                 \sin \varphi & \cos \varphi \\
                                               \end{array}
                                             \right].$$ Clearly, $R_\varphi^{T}=R_\varphi^{-1}=R_{2\pi-\varphi}$ and $\det R_\varphi=1$ for all $\varphi\in[0,2\pi).$ A function $f:\R^{2}\longrightarrow\R$ is said to be \emph{radial} if, for every $x\in\R^{2}$, $f(R_\varphi x)=f(x)$ for all $\varphi\in[0,2\pi)$.
In the rest of this section we will always  assume that the symbol $q(x,\xi)$ of a two-dimensional L\'evy-type process $\process{F}$ is radial in the co-variable.
In particular, if $(b(x), C(x),\nu(x,dy))$ is the corresponding L\'evy triplet, then, due to \cite[Exercise 18.3]{sato-book},
\begin{itemize}
  \item [(i)]$b(x)=0$ for all $x\in\R^{2}$;
  \item [(ii)]$C(x)=c(x)I$ for some Borel measurable function $c:\R^{2}\longrightarrow[0,\infty)$, where $I$ is the $2\times2$ identity matrix;
  \item [(iii)] $\nu(x,dy)=\nu(x,R_\varphi dy)$ for all $x\in\R^{d}$ and all $\varphi\in[0,2\pi).$
\end{itemize}
  Also, let us remark that this assumption implies that the condition in \eqref{eq1.1} does not depend on $r>0$ (see \cite[Proposition 2.4]{sandric-TAMS}). On the other hand, note that if \eqref{eq1.2} holds for some $r_0>0$, then it also holds for all $0<r<r_0.$
According to the same reference, if we need complete independence on $r>0$ in \eqref{eq1.2}, it suffices to assume that the function $\xi\longmapsto\inf_{x\in\R^{2}}\sqrt{q(x,\xi)}$ is subadditive (that is, $\inf_{x\in\R^{2}}\sqrt{q(x,\xi_1+\xi_2)}\leq\inf_{x\in\R^{2}}\sqrt{q(x,\xi_1)}+\inf_{x\in\R^{2}}\sqrt{q(x,\xi_2)}$ for all $\xi_1,\xi_2\in\R^{2}$). The main result of this section is the following (see also \cite{sandric-TAMS} and \cite[Section 38]{sato-book}   for the one-dimensional case).
\begin{theorem}\label{tm3.2} Let $\process{F}$ and $\process{\bar{F}}$ be
two-dimensional  L\'evy-type processes  with  symbols $q(x,\xi)$ and $\bar{q}(x,\xi)$ and  L\'evy triplets
$(0,c(x)I,\nu(x,dy))$ and $(0,\bar{c}(x)I,\bar{\nu}(x,dy))$, respectively.
 If there exists a
rotation of the plane $R_\varphi$, $\varphi\in[0,2\pi)$, such that
\begin{equation}\label{eq3.2}\sup_{x\in\R^{2}}\int_{\R^{2}}
|y|^{2}|\nu(x,dy)-\bar{\nu}(R_\varphi x,dy)|<\infty,\end{equation}  then $q(x,\xi)$ satisfies  \eqref{eq1.1} if, and only if,  $\bar{q}(x,\xi)$ satisfies \eqref{eq1.1}.
Further, denote $$c:=\frac{1}{2}\sup_{x\in\R^{2}}|c(x)-\bar{ c}(R_\varphi x)|+\sup_{x\in\R^{2}}\int_{\R^{2}}
|y|^{2}|\nu(x,dy)-\bar{\nu}(R_\varphi x,dy)|.$$
If
\begin{equation}\label{eq3.3}\liminf_{|\xi|\longrightarrow0}\frac{\inf_{x\in \R^{2}}q(x,\xi)}{|\xi|^{2}}>c,\end{equation}then, under \eqref{eq3.2},
$q(x,\xi)$ satisfies  \eqref{eq1.2} if, and only if,  $\bar{q}(x,\xi)$ satisfies \eqref{eq1.2}.
\end{theorem}

\begin{proof} First, observe that, due to Proposition \ref{p3.1}, $\bar{\nu}(R_\varphi x,dy)=\bar{\nu}(R^{-1}_{2\pi-\varphi} x,R^{-1}_{2\pi-\varphi} dy)$ is the L\'evy measure of the L\'evy-type process $\{R_{2\pi-\varphi} \bar{F}_t\}_{t\geq0}$.
Next, note
that \eqref{eq3.2} implies  $$\sup_{x\in\R^{2}}\int_{\R^{2}}
|y|^{2}\nu(x,dy)<\infty\quad \textrm{if, and only if,}\quad \sup_{x\in\R^{2}}\int_{\R^{2}}
|y|^{2}\bar{\nu}(x,dy)<\infty.$$ Indeed, we have
\begin{align*}\int_{\R^{2}} |y|^{2}\bar{\nu}(R_\varphi x,dy)&=\int_{\R^{2}}
|y|^{2}|\bar{\nu}(R_\varphi x,dy)-\nu(x,dy)+\nu(x,dy)|\\
&\leq\int_{\R^{2}} |y|^{2}\nu(x,dy)+\int_{\R^{2}}
|y|^{2}|\nu(x,dy)-\bar{\nu}(R_\varphi x,dy)|,\end{align*} and similarly
$$\int_{\R^{2}}|y|^{2}\nu(x,dy)\leq\int_{\R^{2}}
|y|^{2}\bar{\nu}(R_\varphi x,dy)+\int_{\R^{2}}
|y|^{2}|\nu(x,dy)-\bar{\nu}(R_\varphi x,dy)|.$$
Now, in the
case when $$\sup_{x\in\R^{2}}\int_{\R^{2}} |y|^{2}\nu(x,dy)<\infty,\quad \textrm{or, equivalently,}\quad \sup_{x\in\R^{2}}\int_{\R^{2}} |y|^{2}\bar{\nu}(x,dy)<\infty,$$  the assertion of the theorem easily follows from \cite[Theorem 2.8]{sandric-TAMS}.
Next, suppose that $$\sup_{x\in\R^{2}}\int_{\R^{2}} |y|^{2}\nu(x,dy)=\infty\quad \textrm{or,
equivalently,} \quad \sup_{x\in\R^{2}}\int_{\R^{2}} |y|^{2}\bar{\nu}(x,dy)=\infty.$$
Then, by using  the radiality  of the function $\xi\longmapsto q(x,\xi)$ and  Fatou's
lemma, we conclude  \begin{align*}\liminf_{|\xi|\longrightarrow0}\frac{\sup_{x\in\R^{2}}q(x,\xi)}{|\xi|^{2}}&=\liminf_{|\xi|\longrightarrow0}\frac{\sup_{x\in\R^{2}}q(x,|\xi|e_1)}{|\xi|^{2}}\\
&\geq\liminf_{|\xi|\longrightarrow0}\sup_{x\in\R^{d}}\int_{\R^{2}}\frac{1-\cos\langle |\xi|e_1,y\rangle}{|\xi|^{2}}\nu(x,dy)\\
&\geq\liminf_{|\xi|\longrightarrow0}\int_{\R^{2}}\frac{1-\cos\langle |\xi|e_1,y\rangle}{|\xi|^{2}}\nu(x,dy)\\
&\geq\frac{1}{2}\int_{\R^{2}}\langle e_1,y\rangle^{2}\nu(x,dy)\\
&=\frac{1}{4}\int_{\R^{2}}\langle e_1,y\rangle^{2}\nu(x,dy)+\frac{1}{4}\int_{\R^{2}}\langle e_2,y\rangle^{2}\nu(x,dy)\\
&=\frac{1}{4}\int_{\R^{2}}|y|^{2}\nu(x,dy),
\end{align*}where $e_1$ and $e_2$ denote the coordinate vectors of $\R^{2}.$
Thus, \begin{equation}\label{eq3.4}\liminf_{|\xi|\longrightarrow0}\frac{\sup_{x\in\R^{2}}q(x,\xi)}{|\xi|^{2}}=\infty.
\end{equation}
Next, we have
\begin{align}\label{eq3.5}&|\sup_{x\in\R^{2}}q(x,\xi)-\sup_{x\in\R^{2}}\bar{q}(x,\xi)|\nonumber\\&
=|\sup_{x\in\R^{2}}q(x,\xi)-\sup_{x\in\R^{2}}\bar{q}(R_\varphi x,\xi)|\nonumber\\&
\leq\sup_{x\in\R^{2}}|q(x,\xi)-\bar{q}(R_\varphi x,\xi)|\nonumber\\&\leq
\frac{1}{2}|\xi|^{2}\sup_{x\in\R^{2}}|c(x)-\bar{ c}(R_\varphi x)| +\sup_{x\in\R^{2}}\left|\int_{\R^{2}}
(1-\cos\langle\xi, y\rangle)\nu(x,dy)-\int_{\R^{2}} (1-\cos\langle\xi,
y\rangle)\bar{\nu}(R_\varphi x,dy)\right|\nonumber\\&\leq
\frac{1}{2}|\xi|^{2}\sup_{x\in\R^{2}}|c(x)-\bar{ c}(R_\varphi x)| +\sup_{x\in\R^{2}}\int_{\R^{2}}
(1-\cos\langle\xi, y\rangle)|\nu(x,dy)-\bar{\nu}(R_\varphi x,dy)|\nonumber\\&\leq\frac{1}{2}|\xi|^{2}\sup_{x\in\R^{2}}|c(x)-\bar{ c}(R_\varphi x)| +|\xi|^{2}\sup_{x\in\R^{2}}\int_{\R^{2}}
|y|^{2}|\nu(x,dy)-\bar{\nu}(R_\varphi x,dy)|\nonumber\\
&=c|\xi|^{2},
\end{align} where in the  the penultimate
 step we used the fact that $1-\cos u\leq u^{2}$
for all $u\in\R.$
Finally, \eqref{eq3.4} and \eqref{eq3.5} imply
$$\lim_{|\xi|\longrightarrow0}\frac{\sup_{x\in\R^{2}}\bar{q}(x,\xi)}{\sup_{x\in\R^{2}}q(x,\xi)}=1+\lim_{|\xi|\longrightarrow0}\frac{\sup_{x\in\R^{2}}\bar{q}(x,\xi)-\sup_{x\in\R^{2}}q(x,\xi)}{\sup_{x\in\R^{2}}q(x,\xi)}=1,$$
which together with the radiality of $\xi\longmapsto q(x,\xi)$ proves the first assertion.

 Now, we prove the second assertion. First, as above,
\begin{align}\label{eq3.6}|\inf_{x\in\R^{2}}q(x,\xi)-\inf_{x\in\R^{2}}\bar{q}(x,\xi)|&=
|\inf_{x\in\R^{2}}q(x,\xi)-\inf_{x\in\R^{2}}\bar{q}(R_\varphi x,\xi)|\nonumber\\&
\leq\sup_{x\in\R^{2}}|q(x,\xi)-\bar{q}(R_\varphi x,\xi)|\nonumber\\&\leq c|\xi|^{2}.
\end{align} Hence, by (\ref{eq3.3}) and (\ref{eq3.6}),
\begin{align*}\liminf_{\xi\longrightarrow0}\frac{\inf_{x\in\R^{2}}\bar{q}(x,\xi)}{\inf_{x\in\R^{2}}q(x,\xi)}&=1+\liminf_{\xi\longrightarrow0}\frac{\inf_{x\in\R^{2}}\bar{q}(x,\xi)-\inf_{x\in\R^{2}}q(x,\xi)}{\inf_{x\in\R^{2}}q(x,\xi)}\\&\geq
1-\frac{c}{\liminf_{|\xi|\longrightarrow0}\frac{\inf_{x\in\R^{2}}q(x,\xi)}{|\xi|^{2}}}\\
&>0,\end{align*}
and
\begin{align*}\limsup_{\xi\longrightarrow0}\frac{\inf_{x\in\R^{2}}\bar{q}(x,\xi)}{\inf_{x\in\R^{2}}q(x,\xi)}&=1+\limsup_{\xi\longrightarrow0}\frac{\inf_{x\in\R^{2}}\bar{q}(x,\xi)-\inf_{x\in\R^{2}}q(x,\xi)}{\inf_{x\in\R^{2}}q(x,\xi)}\\&\leq
1+\frac{c}{\liminf_{|\xi|\longrightarrow0}\frac{\inf_{x\in\R^{2}}q(x,\xi)}{|\xi|^{2}}}\\
&\leq2,\end{align*}
which proves  the desired result.
\end{proof}
Note that every two-dimensional L\'evy process with radial symbol automatically satisfies the relation in \eqref{eq3.3}.
A situation where the
perturbation condition in \eqref{eq3.2} trivially holds true is given in
the following proposition.
\begin{proposition}\label{p3.3}
 Let $\process{F}$ and $\process{\bar{F}}$ be
two-dimensional Le\'vy-type processes with  L\'evy
measures $\nu(x,dy)$ and $\bar{\nu}(x,dy)$, respectively. If there
exist a rotation of the plane $R_\varphi$, $\varphi\in[0,2\pi)$, and $r>0$  such that $\nu(x,B)=\bar{\nu}(R_\varphi x,B)$  for all $x\in\R^{2}$ and all $B\in\mathcal{B}(\R^{2})$, $B\subseteq B^{c}_r(0)$, then  the condition in
\eqref{eq3.2} holds true.
\end{proposition}

Observe that Proposition \ref{p3.3} implies that the recurrence and transience properties of two-dimensional L\'evy-type processes, satisfying the conditions from Theorem \ref{tm3.2}, depend only on big jumps, that is, they do not depend on the continuous part of the process and small jumps. In the following theorem we slightly generalize Proposition \ref{p3.3}.
\begin{theorem} Let $\process{F}$ and $\process{\bar{F}}$ be two-dimensional L\'evy-type
processes with symbols $q(x,\xi)$ and $\bar{q}(x,\xi)$ and L\'evy
triplets $(0,c(x)I,\nu(x,dy))$ and $(0,\bar{c}(x)I,\bar{\nu}(x,dy))$, respectively.  Assume that there exist a rotation of the plane $R_\varphi$, $\varphi\in[0,2\pi)$, and compact set $C\subseteq\R^{2}$, such that
$\nu(x,B)\geq\bar{\nu}(R_\varphi x,B)$ for all $x\in\R^{2}$ and
all $B\in\mathcal{B}(\R^{2})$, $B\subseteq C^{c}$. Then,
\begin{itemize}
  \item [(i)] if $q(x,\xi)$ satisfies \eqref{eq1.1},
 $\bar{q}(x,\xi)$ also satisfies \eqref{eq1.1}.
  \item [(ii)] if
$\bar{q}(x,\xi)$ satisfies   \eqref{eq1.2} and $q(x,\xi)$ satisfies \eqref{eq3.3},
 $q(x,\xi)$ also satisfies \eqref{eq1.2}.
\end{itemize}
\end{theorem}
\begin{proof}First, fix $r>0$  large enough such that $C\subseteq B_r(0)$. Then, by the same reasoning as in the proof of Theorem \ref{tm3.2}, we conclude
$$\sup_{x\in\R^{2}}\bar{q}(x,\xi)=\sup_{x\in\R^{2}}\bar{q}(R_\varphi x,\xi)\leq
\bar{c}|\xi|^{2} +\sup_{x\in\R^{2}}\int_{B^{c}_r(0)}(1-\cos\langle\xi,
y\rangle)\nu(x,dy)
,$$
where
$$\bar{c}=\frac{1}{2}\sup_{x\in\R^{2}}\bar{c}(x)+\sup_{x\in\R^{2}}\int_{B_r(0)}|y|^{2}\bar{\nu}(x,dy).$$
Finally, \eqref{eq3.4} implies
\begin{align*}\limsup_{|\xi|\longrightarrow0}\frac{\sup_{x\in\R^{2}}\bar{q}(x,\xi)}{\sup_{x\in\R^{2}}q(x,\xi)}&\leq\limsup_{|\xi|\longrightarrow0}\frac{\bar{c}|\xi|^{2}+\sup_{x\in\R^{2}}\int_{B_r^{c}(0)}(1-\cos\langle\xi, y\rangle)\nu(x,dy)}{\sup_{x\in\R^{2}}q(x,\xi)}\\&\leq1+\limsup_{|\xi|\longrightarrow0}\frac{\bar{c}|\xi|^{2}}{\sup_{x\in\R^{2}}q(x,\xi)}\\&=1,\end{align*}
which together with the radiality of $\xi\longrightarrow q(x,\xi)$ proves the first assertion.

Now, we prove the second assertion.  Again, fix $r>0$  large enough
such that $C\subseteq B_r(0)$. By the same reasoning as above, we have
$$
\inf_{x\in\R^{2}}\bar{q}(x,\xi)=\inf_{x\in\R^{2}}\bar{q}(R_\varphi x,\xi)\leq \bar{c}|\xi|^{2}+\inf_{x\in\R^{2}}\int_{B_r^{c}(0)}(1-\cos\langle\xi, y\rangle)\nu(x,dy).$$ Thus,
\begin{align*}\limsup_{|\xi|\longrightarrow0}\frac{\inf_{x\in\R^{2}}\bar{q}(x,\xi)}{\inf_{x\in\R^{2}}q(x,\xi)}&\leq\limsup_{|\xi|\longrightarrow0}\frac{\bar{c}|\xi|^{2}+\inf_{x\in\R^{2}}\int_{B_r^{c}(0)}(1-\cos\langle\xi, y\rangle)\nu(x,dy)}{\inf_{x\in\R^{2}}q(x,dy)}\\&\leq1+\limsup_{|\xi|\longrightarrow0}\frac{\bar{c}|\xi|^{2}}{\inf_{x\in\R^{2}}q(x,\xi)}\\&\leq1+\frac{\bar{c}}{c},\end{align*}
 where the constant $c$ is defined in Theorem \ref{tm3.2}.
\end{proof}

\section{Conditions for Recurrence and Transience}\label{s4}

In many cases the Chung-Fuchs type conditions
 are
not practical, that is, it is not always easy to compute (appropriately estimate) the
integrals appearing in \eqref{eq1.1} and \eqref{eq1.2}. According
to this, in the sequel we derive  necessary and sufficient
conditions for the recurrence and transience of two-dimensional
L\'evy-type processes in terms of the L\'evy measures. Throughout this section we again assume  that the symbol $q(x,\xi)$ of a two-dimensional L\'evy-type process $\process{F}$ is radial in the co-variable.
We start this section with the following auxiliary result (see also \cite{sandric-TAMS} and \cite[Section 38]{sato-book}   for the one-dimensional case).
\begin{proposition} \label{tm3.5}Let $\process{F}$  be a
two-dimensional L\'evy-type process with  symbol
$q(x,\xi)$ and L\'evy triplet $(0,c(x)I,\nu(x,dy))$.
Define
$$M\left(x,\rho,u\right):=\nu\left(x,\bigcup_{n=0}^{\infty}(2n\rho+u,2(n+1)\rho-u]\times\R\cap B^{c}_1(0)\right),\quad x\in\R^{2}, \ \rho\geq0,\ 0\leq u\leq\rho.$$ Then,
\begin{equation}\label{eq3.7}\int_r^{\infty}\left(\rho\sup_{x\in\R^{2}}\int_0^{\rho}uM(x,\rho,u)du\right)^{-1}d\rho=\infty\quad \textrm{for some (all)}\ r>0\end{equation}
if, and only if, \eqref{eq1.1} holds true. Further, if
\begin{equation}\label{eq3.8}\liminf_{|\xi|\longrightarrow0}\inf_{x\in\R^{2}}\int_{\R^{2}}\frac{1-\cos\langle\xi,y\rangle}{|\xi|^{2}}\nu(x,dy)=\infty,\end{equation} then
\eqref{eq1.2} holds true if, and only if,
\begin{equation}\label{eq3.9}\int_r^{\infty}\left(\rho\inf_{x\in\R^{2}}\int_0^{r}uM(x,\rho,u)du\right)^{-1}d\rho<\infty\quad \textrm{for some}\ r>0.\end{equation}
\end{proposition}
\begin{proof} First, due to Theorem \ref{tm3.2} and Proposition \ref{p3.3}, without loss of generality, we can assume that $\sup_{x\in\R^{2}}\nu(x,B_1(0))=0$.
Next, note that, because of the radiality of the function $\xi\longmapsto q(x,\xi)$, the conditions in  \eqref{eq1.1} and \eqref{eq1.2} are equivalent to
\begin{equation}\label{eq3.10}\int_{0}^{r}\frac{\rho\,d\rho}{\overline{q}(\rho)}=\infty\quad \textrm{for some (all)}\ r>0\end{equation} and
\begin{equation}\label{eq3.11}\int_{0}^{r}\frac{\rho\,d\rho}{\underline{q}(\rho)}<\infty\quad \textrm{for some}\ r>0,\end{equation} respectively,
where $\overline{q}(\rho):=\sup_{x\in\R^{2}}q(x,\xi)$ and $\underline{q}(\rho):=\inf_{x\in\R^{2}}q(x,\xi)$ for $\rho\in[0,\infty)$ and $\xi\in\R^{2}$, $\rho=|\xi|.$ Denote by $N(x,u):=\nu(x,(u,\infty)\times\R)$ for $x\in\R^{2}$ and
$u\geq0$. Then, by following \cite [Theorem 3.7]{sandric-TAMS}, we have
\begin{align*}q(x,\xi)-\frac{1}{2}c(x)|\xi|^{2}&=\int_{\R^{2}}(1-\cos\langle\xi,
y\rangle)\nu(x,dy)\\&
=\int_{\R^{2}}(1-\cos\langle|\xi|e_i,y\rangle)\nu(x,dy)\\&
=2\int_{(0,\infty)\times\R}(1-\cos\langle|\xi|e_i,y\rangle)\nu(x,dy)\\&=2\int^{\infty}_{0}(1-\cos|\xi| u)d(-N(x,u))
\\&=2|\xi|\int_0^{\infty}N(x,u)\sin|\xi|
u\,du\\&=2|\xi|\sum_{n=0}^{\infty}\int_0^{\frac{2\pi}{|\xi|}}N\left(x,\frac{2\pi
n}{|\xi|}+u\right)\sin|\xi|
u\,du\\&=2|\xi|\sum_{n=0}^{\infty}\left(I_{n,1}+I_{n,2}+I_{n,3}+I_{n,4}\right),\quad i=1,2,\end{align*}
where in the second step we employed the fact that $\nu(x,dy)$ is rotationally invariant, in the fifth step we used the integration by parts formula and
\begin{align*}I_{n,1}&:=\int_0^{\frac{\pi}{2|\xi|}}N\left(x,\frac{2\pi
n}{|\xi|}+u\right)\sin|\xi| u\,du,\\
I_{n,2}&:=\int_{\frac{\pi}{2|\xi|}}^{\frac{\pi}{|\xi|}}N\left(x,\frac{2\pi
n}{|\xi|}+u\right)\sin|\xi|
u\,du=\int_{0}^{\frac{\pi}{2|\xi|}}N\left(x,\frac{2\pi
n}{|\xi|}+\frac{\pi}{|\xi|}-u\right)\sin|\xi| u\,du,\\
I_{n,3}&:=\int_{\frac{\pi}{|\xi|}}^{\frac{3\pi}{2|\xi|}}N\left(x,\frac{2\pi
n}{|\xi|}+u\right)\sin|\xi|
u\,du=-\int_{0}^{\frac{\pi}{2|\xi|}}N\left(x,\frac{2\pi
n}{|\xi|}+\frac{\pi}{|\xi|}+u\right)\sin|\xi| u\,du,\\
I_{n,4}&:=\int_{\frac{3\pi}{2|\xi|}}^{\frac{2\pi}{|\xi|}}N\left(x,\frac{2\pi
n}{|\xi|}+u\right)\sin|\xi|
u\,du=-\int_{0}^{\frac{\pi}{2|\xi|}}N\left(x,\frac{2\pi
n}{|\xi|}+\frac{2\pi}{|\xi|}-u\right)\sin|\xi| u\,du.\end{align*} Thus,
\begin{align*}I_{n,1}+I_{n,4}&=\int_0^{\frac{\pi}{2|\xi|}}\nu\left(x,\left(\frac{2\pi n}{|\xi|}+u,\frac{2\pi(n+1)}{|\xi|}-u\right]\times\R\right)\sin|\xi|
u\,du\\
I_{n,2}+I_{n,3}&=\int_0^{\frac{\pi}{2|\xi|}}\nu\left(x,\left(\frac{\pi
(2n+1)}{|\xi|}-u,\frac{\pi(2n+1)}{|\xi|}+u\right]\times\R\right)\sin|\xi| u\,du.\end{align*}
Now, by defining
$$\bar{M}(x,\rho,u):=\nu\left(x,\bigcup_{n=0}^{\infty}((2n+1)\rho-u,(2n+1)\rho+u]\times\R\right),\quad x\in\R^{2},\ \rho\geq0,\ 0\leq u\leq\rho,$$
$$q(x,\xi)-\frac{1}{2}c(x)|\xi|^{2}=2|\xi|\left(\int_0^{\frac{\pi}{2|\xi|}}M\left(x,\frac{\pi}{|\xi|},u\right)\sin|\xi| u\,du+\int_0^{\frac{\pi}{2|\xi|}}\bar{M}\left(x,\frac{\pi}{|\xi|},u\right)\sin|\xi|
u\,du\right).$$ Further, note that
$$M\left(x,\frac{\pi}{|\xi|},u\right)\geq
\bar{M}\left(x,\frac{\pi}{|\xi|},u\right)\geq0,\quad
u\in\left(0,\frac{\pi}{2|\xi|}\right],$$ and $$\frac{2u}{\pi}\leq\sin
u\leq u,\quad u\in\left(0,\frac{\pi}{2}\right].$$ Thus
\begin{equation}\label{eq3.12}\frac{4}{\pi}|\xi|\int_0^{\frac{\pi}{2|\xi|}}uM\left(x,\frac{\pi}{|\xi|},u\right)du\leq
\frac{q(x,\xi)-
\frac{1}{2}c(x)|\xi|^{2}}{|\xi|}\leq4|\xi|\int_0^{\frac{\pi}{2|\xi|}}uM\left(x,\frac{\pi}{|\xi|},u\right)du.\end{equation}
Next,  we have
\begin{equation*}\int_0^{\frac{\pi}{r}}\left(\varrho\sup_{x\in\R^{2}}\int_0^{\frac{\pi}{\varrho}}uM\left(x,\frac{\pi}{\varrho},u\right)du\right)^{-1}d\varrho=\int_r^{\infty}\left(\rho\sup_{x\in\R^{2}}\int_0^{\rho}uM\left(x,\rho,u\right)du\right)^{-1}d\rho\end{equation*} and
\begin{equation*}\int_0^{\frac{\pi}{r}}\left(\varrho\inf_{x\in\R^{2}}\int_0^{\frac{\pi}{r}}uM\left(x,\frac{\pi}{\varrho},u\right)du\right)^{-1}d\varrho=\int_r^{\infty}\left(\rho\inf_{x\in\R^{2}}\int_0^{\rho}uM\left(x,\rho,u\right)du\right)^{-1}d\rho,\end{equation*}
where we use the change of variables $\varrho\longmapsto\pi/\rho$. Thus,
\eqref{eq3.7} implies
$$\int_0^{\frac{\pi}{r}}\frac{\rho\, d\rho}{\sup_{x\in\R^{2}}\left(q(x,(\rho,0))-\frac{1}{2}c(x)\rho^{2}\right)}=\infty$$ and
$$\int_0^{\frac{\pi}{r}}\frac{\rho\, d\rho}{\inf_{x\in\R^{2}}\left(q(x,(\rho,0))-\frac{1}{2}c(x)\rho^{2}\right)}<\infty$$ implies \eqref{eq3.9}.
Finally,
\begin{align}\label{eq3.13}1&\leq\liminf_{\rho\longrightarrow0}\frac{\overline{q}(\rho)}{\sup_{x\in\R^{2}}\left(q(x,(\rho,0))-\frac{1}{2}c(x)\rho^{2}\right)}\nonumber\\
&\leq\limsup_{\rho\longrightarrow0}\frac{\overline{q}(\rho)}{\sup_{x\in\R^{2}}\left(q(x,(\rho,0))-\frac{1}{2}c(x)\rho^{2}\right)}\nonumber\\
&\leq\limsup_{\rho\longrightarrow0}\frac{\frac{1}{2}\rho^{2}\sup_{x\in\R^{2}}c(x)+\sup_{x\in\R^{2}}\left(q(x,(\rho,0))-\frac{1}{2}c(x)\rho^{2}\right)}{\sup_{x\in\R^{2}}\left(q(x,(\rho,0))-\frac{1}{2}c(x)\rho^{2}\right)} \nonumber\\
&\leq1+ \frac{\frac{1}{2}\sup_{x\in\R^{2}}c(x)}{\liminf_{\rho\longrightarrow0}\sup_{x\in\R^{2}}\int_{\R^{2}}\frac{1-\cos\langle(\rho,0), y\rangle}{\rho^{2}}\nu(x,dy)}\nonumber\\
&\leq1+ 2\frac{\sup_{x\in\R^{2}}c(x)}{\sup_{x\in\R^{2}}\int_{\R^{2}}|y|^{2}\nu(x,dy)},\end{align}
where in the final step we employed Fatou's lemma and the fact that $\nu(x,dy)$ is rotationally invariant.
Analogously,
\begin{align}\label{eq3.14}1&\leq\liminf_{\rho\longrightarrow0}\frac{\underline{q}(\rho)}{\left(q(x,(\rho,0))-\frac{1}{2}c(x)\rho^{2}\right)}\nonumber\\
&\leq\limsup_{\rho\longrightarrow0}\frac{\underline{q}(\rho)}{\inf_{x\in\R^{2}}\left(q(x,(\rho,0))-\frac{1}{2}c(x)\rho^{2}\right)}\nonumber\\
&\leq
1+ \frac{\frac{1}{2}\sup_{x\in\R^{2}}c(x)}{\liminf_{\rho\longrightarrow0}\inf_{x\in\R^{2}}\int_{\R^{2}}\frac{1-\cos\langle(\rho,0), y\rangle}{\rho^{2}}\nu(x,dy)}.\end{align}
Now, the assertion follows from \eqref{eq3.8}.

To prove the converse, first note that
\begin{align*}&\int_{0}^{\frac{\pi}{|\xi|}}u\,\nu\left(x,\left(\frac{2n\pi}{|\xi|}+u,\frac{2(n+1)\pi}{|\xi|}-u\right]\times\R\right)du\\&=
\int_{0}^{\frac{\pi}{2|\xi|}}u\,\nu\left(x,\left(\frac{2n\pi}{|\xi|}+u,\frac{2(n+1)\pi}{|\xi|}-u\right]\times\R\right)du\nonumber\\&\ \ \ +\int_{\frac{\pi}{2|\xi|}}^{\frac{\pi}{|\xi|}}u\,\nu\left(x,\left(\frac{2n\pi}{|\xi|}+u,\frac{2(n+1)\pi}{|\xi|}-u\right]\times\R\right)du\\&=
\int_{0}^{\frac{\pi}{2|\xi|}}u\,\nu\left(x,\left(\frac{2n\pi}{|\xi|}+u,\frac{2(n+1)\pi}{|\xi|}-u\right]\times\R\right)du\nonumber\\&\ \ \ +4\int_{\frac{\pi}{4|\xi|}}^{\frac{\pi}{2|\xi|}}u\,\nu\left(x,\left(\frac{2n\pi}{|\xi|}+2u,\frac{2(n+1)\pi}{|\xi|}-2u\right]\times\R\right)du\\&\leq
\int_{0}^{\frac{\pi}{2|\xi|}}u\,\nu\left(x,\left(\frac{2n\pi}{|\xi|}+u,\frac{2(n+1)\pi}{|\xi|}-u\right]\times\R\right)du\nonumber\\ &\ \ \ +4\int_{\frac{\pi}{4|\xi|}}^{\frac{\pi}{2|\xi|}}u\,\nu\left(x,\left(\frac{2n\pi}{|\xi|}+u,\frac{2(n+1)\pi}{|\xi|}-u\right]\times\R\right)du\\&\leq
5\int_{0}^{\frac{\pi}{2|\xi|}}u\,\nu\left(x,\left(\frac{2n\pi}{|\xi|}+u,\frac{2(n+1)\pi}{|\xi|}-u\right]\times\R\right)du.\end{align*}
Hence,
$$\int_{0}^{\frac{\pi}{|\xi|}}uM\left(x,\frac{\pi}{|\xi|},u\right)du\leq5\int_{0}^{\frac{\pi}{2|\xi|}}uM\left(x,\frac{\pi}{|\xi|},u\right)du,$$ that is,
\begin{align*}
\int_r^{\infty}\left(\rho\sup_{x\in\R^{2}}\int_0^{\rho}uM\left(x,\rho,u\right)du\right)^{-1}d\rho
&=\int_0^{\frac{\pi}{r}}\left(\varrho\sup_{x\in\R^{2}}\int_0^{\frac{\pi}{\varrho}}uM\left(x,\frac{\pi}{\varrho},u\right)du\right)^{-1}d\varrho\\
&\geq\frac{1}{5}\int_0^{\frac{\pi}{r}}\left(\varrho\sup_{x\in\R^{2}}\int_0^{\frac{\pi}{2\varrho}}uM\left(x,\frac{\pi}{\varrho},u\right)du\right)^{-1}d\varrho\end{align*}
and
\begin{align*}
\int_r^{\infty}\left(\rho\inf_{x\in\R^{2}}\int_0^{\rho}uM\left(x,\rho,u\right)du\right)^{-1}d\rho
&=\int_0^{\frac{\pi}{r}}\left(\varrho\inf_{x\in\R^{2}}\int_0^{\frac{\pi}{\varrho}}uM\left(x,\frac{\pi}{\varrho},u\right)du\right)^{-1}d\varrho\\
&\geq\frac{1}{5}\int_0^{\frac{\pi}{r}}\left(\varrho\inf_{x\in\R^{2}}\int_0^{\frac{\pi}{2\varrho}}uM\left(x,\frac{\pi}{\varrho},u\right)du\right)^{-1}d\varrho,\end{align*}
where in the first steps we again used the change of variables
$\rho\longmapsto\pi/\varrho$. Thus, according to  \eqref{eq3.8}, \eqref{eq3.12}, \eqref{eq3.13}
and  \eqref{eq3.14},      \eqref{eq3.10} and \eqref{eq3.9} imply \eqref{eq3.7} and \eqref{eq3.11},
respectively.
\end{proof}
Observe that in the L\'evy process case the condition in \eqref{eq3.8} is trivially satisfied.
As a direct consequence of  Proposition \ref{tm3.5}, we  get the following
characterization of the recurrence and transience  in terms of the
tail behavior of the L\'evy measures.
\begin{theorem}\label{c3.6} Let $\process{F}$  be a
two-dimensional L\'evy-type process with  symbol
$q(x,\xi)$ and  L\'evy measure $\nu(x,dy)$. Define
$$N\left(x,u\right):=\nu\left(x,(u,\infty)\times\R\right),\quad x\in\R^{2},\ u\geq0.$$  Then,
\begin{align}\label{eq3.15}\int_r^{\infty}\left(\rho\sup_{x\in\R^{2}}\int_0^{\rho}uN(x,u)du\right)^{-1}d\rho=\infty\quad \textrm{for some}\ r>0\end{align}
implies \eqref{eq3.7}, and \eqref{eq3.9} implies
\begin{align}\label{eq3.16}\int_r^{\infty}\left(\rho\inf_{x\in\R^{2}}\int_0^{\rho}uN(x,u)du\right)^{-1}d\rho<\infty\quad \textrm{for some}\ r>0.\end{align}
\end{theorem}
\begin{proof} The assertion  follows directly from the fact
$N(x,u)\geq M(x,\rho,u)$ for all $x\in\R^{2}$, all $\rho\geq0$ and all $0\leq u\leq \rho$.
\end{proof}
In general, we cannot conclude the equivalence in Theorem \ref{c3.6} (see \cite[Theorem 38.4]{sato-book}).
However, if, in addition, we assume the quasi-unimodality of the
measure $\nu(x,du\times\R)$, $x\in\R^{2}$,  then \eqref{eq3.7} will be equivalent to \eqref{eq3.15} and \eqref{eq3.9} will be equivalent to \eqref{eq3.16}. Recall that a symmetric Borel measure
$\mu(dx)$ on $\mathcal{B}(\R)$ is \emph{quasi-unimodal} if there
exists $x_0\geq0$ such that $x\longmapsto\mu(x,\infty)$ is a convex
function on $(x_0,\infty)$. Equivalently, a symmetric Borel measure
$\mu(dx)$ on $\mathcal{B}(\R)$ is quasi-unimodal if it is of the
form $\mu(dx)=\mu_0(dx)+f(x)dx,$ where the measure $\mu_0(dx)$ is
supported on $[-x_0,x_0]$, for some $x_0\geq0$, and the density
function $f(x)$ is supported on $[-x_0,x_0]^{c}$, it is symmetric
and decreasing on $(x_0,\infty)$ and
$\int_{x_0+\varepsilon}^{\infty}f(x)dx<\infty$ for every
$\varepsilon>0$ (see \cite[Chapters 5 and 7]{sato-book}). When
$x_0=0$, then $\mu(dx)$ is said to be \emph{unimodal}.

\begin{theorem}\label{tm3.7}
Let $\process{F}$  be a
two-dimensional L\'evy-type process with  symbol
$q(x,\xi)$ and L\'evy triplet $(0,c(x)I,\nu(x,dy))$. Assume  that there exists $u_0>1$ such that
the measure $\nu(x,du\times\R)$ is quasi-unimodal with respect to $u_0$ for all  $x\in\R^{2}$.
Then,
\eqref{eq1.1} holds true if, and only if, \eqref{eq3.15} holds
true. Further, if, in addition,
 $\nu(x,dy)$ satisfies \eqref{eq3.8},
 then \eqref{eq1.2} holds true if, and
only if, \eqref{eq3.16} holds true.
\end{theorem}
\begin{proof} According to  Theorem \ref{c3.6}, we only have to prove  that
\eqref{eq1.1} implies \eqref{eq3.15} and that \eqref{eq3.16} implies \eqref{eq1.2}.
First, we prove that \eqref{eq1.1} implies \eqref{eq3.15}.
Due to the radiality of the function $\xi\longmapsto q(x,\xi)$, \eqref{eq1.1} is equivalent to
\begin{equation}\label{eq3.19}\int_0^{r}\frac{\rho\, d\rho}{\sup_{x\in\R^{2}}q(x,(\rho,0))}=\infty\quad \textrm{for some (all)}\ r>0.\end{equation}
Further, we have
\begin{align*}1&\leq\liminf_{\rho\longrightarrow0}\frac{\sup_{x\in\R^{2}}q(x,(\rho,0))}{\sup_{x\in\R^{2}}\int_{\R^{2}}(1-\cos\langle(\rho,0),y\rangle)\nu(x,dy)}\\
&\leq\limsup_{\rho\longrightarrow0}\frac{\sup_{x\in\R^{2}}q(x,(\rho,0))}{\sup_{x\in\R^{2}}\int_{\R^{2}}(1-\cos\langle(\rho,0),y\rangle)\nu(x,dy)}\\
&\leq1+\limsup_{\rho\longrightarrow0}\frac{\frac{1}{2}\sup_{x\in\R^{2}}c(x)}{\sup_{x\in\R^{2}}\int_{\R^{2}}\frac{1-\cos\langle(\rho,0),y\rangle}{\rho^{2}}\nu(x,dy)}\\
&\leq1+2\frac{\sup_{x\in\R^{2}}c(x)}{\sup_{x\in\R^{2}}\int_{\R^{2}}|y|^{2}\nu(x,dy)},\end{align*}
which, together with \cite[Proposition 2.4]{sandric-TAMS}, yields that \eqref{eq3.19} is equivalent to
\begin{equation}\label{eq3.20} \int_0^{r}\frac{\rho\, d\rho}{\sup_{x\in\R^{2}}\int_{\R^{2}}(1-\cos\langle(\rho,0),y\rangle)\nu(x,dy)}=\infty\quad\textrm{for some (all)}\ r>0.\end{equation} Now,  define $\bar{\nu}(x,dy):=\nu(x,dy\cap B^{c}_{u_0}(0)),$ $x\in\R^{2}$. Obviously, $\bar{\nu}(x,dy)$ is rotationally invariant,  $\sup_{x\in\R^{2}}\bar{\nu}(x,\R^{2})<\infty$ and
$$\sup_{x\in\R^{2}}\int_{\R^{2}}|y|^{2}|\nu(x,dy)-\bar{\nu}(x,dy)|\leq\sup_{x\in\R^{2}}\int_{B_{u_0}(0)}|y|^{2}\nu(x,dy)<\infty.$$
Thus, by analogues arguments as in the proof of Theorem \ref{tm3.2}, it is easy to see that \eqref{eq3.20} is equivalent to
$$\int_0^{r}\frac{\rho\,d\rho}{\sup_{x\in\R^{2}}\int_{\R^{2}}(1-\cos\langle(\rho,0),y\rangle)\bar{\nu}(x,dy)}=\infty\quad\textrm{for some (all)}\ r>0,$$ that is,
\begin{equation}\label{eq3.21}
\int_0^{r}\frac{\rho\,d\rho}{\sup_{x\in\R^{2}}\int_{\R}(1-\cos\rho u)\bar{\nu}(x,du\times\R)}=\infty\quad \textrm{for some (all)}\ r>0.\end{equation}
 Next, due to  \cite[the proof of Theorem 3.9]{sandric-TAMS}, for every $x\in\R^{2}$ there exists an unimodal probability measure $\eta_U(x,du)$ on $\mathcal{B}(\R)$ such that $\bar{\nu}(x,(u,\infty)\times\R)=c\eta_U(x,(u,\infty))$ for all $x\in\R^{2}$ and all $u\geq u_0+1$, where   $c:=c(u_0)$ is an appropriate norming constant. Note that
$$\sup_{x\in\R^{2}}\int_{\R}u^{2}|\bar{\nu}(x,du\times\R)-c\eta_U(x,du)|<\infty.$$ According to this, \cite[Theorem 3.1]{sandric-TAMS} implies that \eqref{eq3.21} is equivalent to
\begin{equation}\label{eq3.22}
\int_0^{r}\frac{\rho\,d\rho}{\sup_{x\in\R^{2}}\int_{\R}(1-\cos\rho u)\eta_U(x,du)}=\infty\quad \textrm{for some (all)}\ r>0.\end{equation}
In the sequel we prove that \eqref{eq3.22} implies \eqref{eq3.15}.
First, since $\eta_U(x,du)$ is unimodal, by \cite[Exercise
29.21]{sato-book}, there exists a random variable $X_x$ such that
$\eta_U(x,du)$ is the distribution of the random variable $U_xX_x$,
where $U_x$ is  uniformly distributed random variable on $[0,1]$
independent of $X_x$. Further, let $\eta(x,du)$ be the
distribution of the random variable $X_x$. By \cite[Lemma
38.6]{sato-book},
$\eta(x,(u,\infty))\geq\eta_U(x,(u,\infty))$ for all
$x\in\R^{2}$ and all $u\geq0$.
  Now, we have
  $$\int_\R(1-\cos\rho u)\eta_U(x,du)=\int_0^{1}\int_{\R}(1-\cos(\rho uv))\eta(x,du)dv=\int_\R\left(1-\frac{\sin\rho u}{\rho u}\right)\eta(x,du).$$
 Next,    since $$1-\frac{\sin u}{u}\geq \bar{c}\min\{1, u^{2}\}$$ for
 all $u\in\R$ and all $0<\bar{c}<\frac{1}{6}$,
               $$\int_\R(1-\cos\rho u)\eta_U(x,du)\geq\bar{c}\int_\R\min\{1,\rho^{2} u^{2}\}\eta(x,du)=4\bar{c}\rho^{2}\int_0^{\frac{1}{|\rho|}}uE(x,u)du,$$
where $E(x,u):=\eta(x,(u,\infty))$ for $x\in\R^{2}$ and $u\geq0$. Set
$E_U(x,u):=\eta_U(x,(u,\infty))$ for $x\in\R^{2}$ and $u\geq0$.
Then, for any $r>0$, we have
\begin{align}\label{eq3.23}\int_{r}^{\infty}\left(\rho\sup_{x\in\R^{2}}\int_0^{\rho}uE_U(x,u)du\right)^{-1}d\rho&\geq\int_{r}^{\infty}\left(\rho\sup_{x\in\R^{2}}\int_0^{\rho}uE(x,u)du\right)^{-1}d\rho\nonumber\\&=\int_{0}^{\frac{1}{r}}\left(\rho\sup_{x\in\R^{2}}\int_0^{\frac{1}{\rho}}uE(x,u)du\right)^{-1}d\rho\nonumber\\&\geq4\bar{c}\int_{0}^{\frac{1}{r}}\frac{\rho \,d\rho}{\sup_{x\in\R^{2}}\int_{\R}(1-\cos\rho u)\eta_U(x,du)}.\end{align}
Further,
\begin{align*}&\lim_{\rho\longrightarrow\infty}\frac{\sup_{x\in\R^{2}}\int_{u_0}^{\rho}uN(x,u)du}{\sup_{x\in\R^{2}}\int_0^{u_0}uE_U(x,u)du+\frac{1}{c}\sup_{x\in\R^{2}}\int_{u_0}^{\rho}uN(x,u)du}\\&\leq\liminf_{\rho\longrightarrow\infty}\frac{\sup_{x\in\R^{2}}\int_0^{\rho}uN(x,u)du}{\sup_{x\in\R^{2}}\int_0^{\rho}uE_U(x,u)du}\\&\leq\limsup_{\rho\longrightarrow\infty}\frac{\sup_{x\in\R^{2}}\int_0^{\rho}uN(x,u)du}{\sup_{x\in\R^{2}}\int_0^{\rho}uE_U(x,u)du}\\&\leq\lim_{\rho\longrightarrow\infty}\frac{\sup_{x\in\R^{2}}\int_0^{u_0}uN(x,u)du+\sup_{x\in\R^{2}}\int_{u_0}^{\rho}uN(x,u)du}{\frac{1}{c}\sup_{x\in\R^{2}}\int_{u_0}^{\rho}uN(x,u)du}.\end{align*}
Now, if $\sup_{x\in\R^{2}}\int_{u_0}^{\infty}uN(x,u)du=0$ the desired
result trivially follows. On the other hand, if
$\sup_{x\in\R^{2}}\int_{u_0}^{\infty}uN(x,u)du>0$, we have
\begin{equation}\label{eq3.24}0<\liminf_{\rho\longrightarrow\infty}\frac{\sup_{x\in\R^{2}}\int_0^{\rho}uN(x,u)du}{\sup_{x\in\R^{2}}\int_0^{\rho}uE_U(x,u)du}\leq\limsup_{\rho\longrightarrow\infty}\frac{\sup_{x\in\R^{2}}\int_0^{\rho}uN(x,u)du}{\sup_{x\in\R^{2}}\int_0^{\rho}uE_U(x,u)du}<\infty,\end{equation}
which together with \eqref{eq3.23} proves the assertion.

Finally, we prove that
 \eqref{eq3.16} implies \eqref{eq1.2}. Due to the radiality of $\xi\longmapsto q(x,\xi)$ and \eqref{eq3.8}, by completely the same arguments as above,
$$4\bar{c}\int_{0}^{\frac{1}{r}}\frac{\rho\,d\rho}{\inf_{x\in\R^{2}}\int_{\R}(1-\cos\rho u)\eta_U(x,du)}\leq\int_{r}^{\infty}\left(\rho\inf_{x\in\R^{2}}\int_0^{\rho}uE_U(x,u)du\right)^{-1}d\rho.$$
Now, the desired result follows by a similar argumentation as in \eqref{eq3.24} and employing \eqref{eq3.8}.
\end{proof}
Note that the measure $\nu(x,du\times\R)$ will be quasi-unimodal uniformly in $x\in\R^{2}$ if there exists $u_0\geq0$ such that $\nu(x,dy)=n(x,|y|)dy$ on $\mathcal{B}(B^{c}_{u_0}(0))$ for some Borel function $n:\R^{2}\times(0,\infty)\longrightarrow(0,\infty)$ which is decreasing on $(u_0,\infty)$ for all $x\in\R^{2}.$ Also, let us remark
that
in the L\'evy process case the condition in  \eqref{eq3.8} will be satisfied if, and only if, $\int_{\R^{2}}|y|^{2}\nu(dy)=\infty.$ Recall that $\int_{\R^{2}}|y|^{2}\nu(dy)<\infty$ implies that the underlying L\'evy process is recurrent (see \cite[Theorem 2.8]{sandric-TAMS}).

Finally, as a direct consequence of  Theorem  \ref{tm3.7}, we  get  the following
characterization of the recurrence and transience  in terms of
 the tail behavior of the L\'evy measures.
\begin{theorem}\label{tm4}
Let $\process{F}$  be a
two-dimensional L\'evy-type process with   L\'evy measure $\nu(x,dy)$, satisfying the assumptions from Theorem  \ref{tm3.7}.
Then,
\eqref{eq1.1} holds true if, and only if,
\begin{equation}\label{eq5}\int_r^{\infty}\left(\rho\sup_{x\in\R^{2}}\int_0^{\rho}u\,\nu(x,B_u^{c}(0))du\right)^{-1}d\rho=\infty\quad\textrm{for some (all)}\ r>0,\end{equation}
and \eqref{eq1.2} holds true if, and
only if, \begin{equation}\label{eq6}\int_r^{\infty}\left(\rho\inf_{x\in\R^{2}}\int_0^{\rho}u\,\nu(x,B_u^{c}(0))du\right)^{-1}d\rho<\infty\quad\textrm{for some}\ r>0.\end{equation}
\end{theorem}
\begin{proof}
The assertion is a direct consequence of the following simple fact $$\frac{1}{4}\nu(x,B_{\sqrt{2}u}^{c}(0))\leq N(x,u)\leq\nu(x,B_u^{c}(0)),\quad x\in\R^{2},\ u>0.$$
\end{proof}

\begin{proposition}\label{p5}Let  $\process{F}$  be a
two-dimensional  L\'evy-type process with   L\'evy measure  $\nu(x,dy)$.
Then,
\eqref{eq5} holds  if
$$\int_{r}^{\infty}\left(\rho^{3}\sup_{x\in\R^{2}}\nu(x,B^{c}_\rho(0))+\rho\sup_{x\in\R^{2}}\int_{B_\rho(0)}|y|^{2}\,\nu(x,dy)\right)^{-1}d\rho=\infty\quad\textrm{for some}\ r>0,$$ and \eqref{eq6} holds  if $$\int_{r}^{\infty}\left(\rho^{3}\inf_{x\in\R^{2}}\nu(x,B^{c}_\rho(0))+\rho\inf_{x\in\R^{2}}\int_{B_\rho(0)}|y|^{2}\,\nu(x,dy)\right)^{-1}d\rho<\infty\quad\textrm{for some}\ r>0.$$ In particular, \eqref{eq6} holds if either one of the following  conditions holds
$$\int_{r}^{\infty}\left(\rho^{3}\inf_{x\in\R^{2}}\nu(x,B^{c}_\rho(0))\right)^{-1}d\rho<\infty\quad\textrm{for some}\ r>0$$ or
$$\int_{r}^{\infty}\left(\rho\inf_{x\in\R^{2}}\int_{B_\rho(0)}|y|^{2}\,\nu(x,dy)\right)^{-1}d\rho<\infty\quad\textrm{for some}\ r>0.$$
In addition, if $\nu(x,dy)$ is of the form  $\nu(x,dy)=n(x,|y|)dy$, where $n:\R^{2}\times(0,\infty)\longrightarrow(0,\infty)$ is a Borel function, and
there exists $u_0\geq0$ such that $n(x,u)$ is decreasing on $(u_0,\infty)$ for all $x\in\R^{2}$, then \eqref{eq6} holds if
$$\int_{r}^{\infty}\frac{du}{u^{5}\inf_{x\in\R^{2}}n(x,u)}<\infty\quad \textrm{for some}\ r\geq u_0.$$
\end{proposition}
\begin{proof}
 By employing the integration by parts formula, for any $x\in\R^{2}$, any $\rho>0$ and any $0<\varepsilon<\rho$, we have
\begin{align*}&\frac{\rho^{2}}{2}\nu(x,B^{c}_\rho(0))+\frac{1}{2}\int_{B_\rho(0)}|y|^{2}\,\nu(x,dy)\\
&\geq\int_0^{\rho}u\,\nu(x,B^{c}_u(0))du\\&=\int_0^{\varepsilon}u\,\nu(x,B^{c}_u(0))du+\int_{\varepsilon}^{\rho}u\,\nu(x,B^{c}_u(0))du\\
&\geq\frac{\varepsilon^{2}}{2}\nu(x,B^{c}_{\varepsilon}(0))+\frac{\rho^{2}}{2}\nu(x,B^{c}_\rho(0))-\frac{\varepsilon^{2}}{2}\nu(x,B^{c}_{\varepsilon}(0))+\frac{1}{2}\int_{B_\rho(0)\cap´B^{c}_{\varepsilon}(0)}|y|^{2}\,\nu(x,dy)\\
&=\frac{\rho^{2}}{2}\nu(x,B^{c}_\rho(0))+\frac{1}{2}\int_{B_\rho(0)\cap´B^{c}_{\varepsilon}(0)}|y|^{2}\,\nu(x,dy).\end{align*} Now, by letting $\varepsilon\longrightarrow0$, Fatou's lemma yields $$\int_0^{\rho}u\,\nu(x,B^{c}_u(0))du=\frac{\rho^{2}}{2}\nu(x,B^{c}_\rho(0))+\frac{1}{2}\int_{B_\rho(0)}|y|^{2}\,\nu(x,dy).$$

In order to prove the last assertion, observe that for any $r>u_0$,  \begin{align*}\int_{r}^{\infty}\left(\rho\inf_{x\in\R^{2}}\int_{B_\rho(0)}|y|^{2}\,\nu(x,dy)\right)^{-1}d\rho&\leq\frac{1}{2\pi}\int_{r}^{\infty}\left(\rho\inf_{x\in\R^{2}}\int_{u_0}^{\rho}u^{3}n(x,u)du\right)^{-1}d\rho\\&\leq\frac{2}{\pi}\int_{r}^{\infty}\frac{d\rho}{\rho(\rho^{4}-u_0^{4})\inf_{x\in\R^{2}}n(x,\rho)}\\&\leq
c\int_{r}^{\infty}\frac{d\rho}{\rho^{5}\inf_{x\in\R^{2}}n(x,\rho)},\end{align*}
where in the second step we employed the fact that $n(x,u)$ is decreasing  on $(u_0,\infty)$ and $c>\frac{2r^{4}}{\pi(r^{4}-u_0^{4})}$ is
arbitrary.
\end{proof}
Observe  that from the previous proposition we again conclude that  if $\sup_{x\in\R^{2}}\int_{\R^{2}}|y|^{2}\nu(x,dy)<\infty$, then $\nu(x,dy)$ automatically satisfies \eqref{eq5}. Let us now give some applications of the results presented above.

\begin{example}\label{e1}{\rm
 Let $\alpha:\R^{2}\longrightarrow(0,2)$ and
$\beta:\R^{2}\longrightarrow(0,\infty)$ be arbitrary bounded and
continuously differentiable functions with bounded derivatives, such
that
$0<\inf_{x\in\R^{2}}\alpha(x)\leq\sup_{x\in\R^{2}}\alpha(x)<2$
and $\inf_{x\in\R^{2}}\beta(x)>0$. Under this assumptions, in
\cite{bass-stablelike}, \cite[Theorem 5.1]{Kolokoltsov-2000} and
                                                         \cite[Theorem 3.3.]{rene-wang-feller}
                                                         it has been shown
                                                         that there
                                                         exists a
                                                         unique open-set irreducible L\'evy-type
                                                         process $\process{F}$,
                                                         called a (two-dimensional)
                                                         \emph{stable-like
                                                         process},
                                                         determined
by a L\'evy triplet and symbol of the form $(0,0,\beta(x)|y|^{-2-\alpha(x)}dy)$ and
$q(x,\xi)=\gamma(x)|\xi|^{\alpha(x)}$, respectively, where
$$\gamma(x):=\beta(x)\frac{\pi^{1/2}\Gamma(1-\alpha(x)/2)}{\alpha(x)
                 2^{\alpha(x)-1}\Gamma((\alpha(x)+1)/2)},\quad x\in\R^{2}.$$ Here, $\Gamma(x)$ denotes the Gamma function.
Note that when
$\alpha(x)$ and $\beta(x)$ are constant functions, then we deal
with a rotationally invariant two-dimensional stable L\'evy process.
 Now, by a direct application of the Chung-Fuchs type condition in \eqref{eq1.2} we easily see that $\process{F}$ is transient. On the other hand,  the corresponding L\'evy measure also satisfies all the assumptions from Theorem \ref{tm4}, which  again implies the transience of $\process{F}$.
For more on stable-like processes and their recurrence and transience properties we refer the readers to \cite{bass-stablelike}, \cite{bjoern-overshoot}, \cite{franke-periodic, franke-periodicerata},
  \cite{sandric-spa}, \cite{sandric-TAMS} and \cite{rene-wang-feller}.}
\end{example}

\begin{example}\label{e2}{\rm Let $\alpha,\beta:\R^{2}\longrightarrow(0,\infty)$ and $\gamma:\R^{2}\longrightarrow\R$ be arbitrary bounded and
continuous functions such
that
$\inf_{x\in\R^{2}}\alpha(x)>0$ and $\inf_{x\in\R^{2}}\beta(x)>0$. Define $n:\R^{2}\times(0,\infty)\longrightarrow(0,\infty)$ by $$n(x,u):=\frac{\beta(x)\ln^{\gamma(x)}u}{u^{2+\alpha(x)}}1_{\{v:v\geq e\}}(u).$$ Because of the continuity of $\alpha(x)$, $\beta(x)$ and $\gamma(x)$, without loss of generality, we can assume that $\int_{\R^{2}}n(x,|y|)dy=1$ for all $x\in\R^{2}.$ Now, by (a straightforward adaptation of) \cite[Proposition 2.9]{Sandric-JOTP-2014}, there exists a unique open-set irreducible L\'evy-type process $\process{F}$ determined by a L\'evy measure and symbol of the form $\nu(x,dy):=n(x,|y|)dy$ and   $q(x,\xi):=\int_{\R^{2}}(1-\cos\langle\xi,y\rangle)\nu(x,dy)$, respectively.  Put $\underline{\alpha}:=\inf_{x\in\R^{2}}\alpha(x)\leq\sup_{x\in\R^{2}}\alpha(x)=:\overline{\alpha}$, $\underline{\beta}:=\inf_{x\in\R^{2}}\beta(x)\leq\sup_{x\in\R^{2}}\beta(x)=:\overline{\beta}$ and
 $\underline{\gamma}:=\inf_{x\in\R^{2}}\gamma(x)\leq\sup_{x\in\R^{2}}\gamma(x)=:\overline{\gamma}.$
  Then,
 \begin{enumerate}
   \item [(i)] if $\underline{\alpha}>2$ (which automatically implies that $\sup_{x\in\R^{2}}\int_{\R^{2}}|y|^{2}\nu(x,dy)
<\infty$),  a  direct application of the Chung-Fuchs type  condition in \eqref{eq1.1} (or Proposition \ref{p5})   entails the recurrence of $\process{F}$.
   \item [(ii)] if $\overline{\alpha}<2$, since $$n(x,u)\geq\frac{\underline{\beta}}{u^{2+\overline{\alpha}+\varepsilon}}1_{\{v:v\geq e\}}(u)$$ for some $0<\varepsilon<2-\overline{\alpha}$, all $x\in\R^{2}$ and all $u>0$ large enough,
   Theorem \ref{tm3.2}  and Example \ref{e1} (or Proposition \ref{p5})  imply that $\process{F}$ is transient.
 \end{enumerate}
  On the other hand, in order to conclude the recurrence or transience  of  $\process{F}$ in the cases when  $\underline{\alpha}=2$ or $\overline{\alpha}=2$,
 it is not immediately  clear how to (explicitly) compute or (appropriately) bound its symbol
and apply the Chung-Fuchs type conditions. However, since $\nu(x,dy)$ obviously satisfies all the assumptions of Theorem \ref{tm4},  we conclude that
\begin{enumerate}
\item [(iii)] if $\underline{\alpha}\geq2$ and $\overline{\gamma}\leq0$,  then
  $$\int_r^{\infty}\left(\rho\sup_{x\in\R^{2}}\int_0^{\rho}u\,\nu(x,B^{c}_u(0))\,du\right)^{-1}d\rho\geq c(\underline{\alpha},\overline{\beta})\int_r^{\infty}\frac{d\rho}{\rho\ln\rho},\quad r\geq e,$$ which  entails the recurrence of $\process{F}.$
\item [(iv)]if $\overline{\alpha}\leq2$ and $\underline{\gamma}>0$, then
  $$\int_r^{\infty}\left(\rho\inf_{x\in\R^{2}}\int_0^{\rho}u\,\nu(x,B^{c}_u(0))\,du\right)^{-1}d\rho\leq c(\overline{\alpha},\underline{\beta},\underline{\gamma})\int_{ r}^{\infty}\frac{d\rho}{\rho\ln^{\underline{\gamma}+1}\rho},\quad r\geq e,$$which implies that $\process{F}$ is transient.
\end{enumerate}
}\end{example}
\begin{example}\label{e3}{\rm
Let $\process{L}$ be a L\'evy process with L\'evy measure of the form $\nu(dy)=n(|y|)dy$, where $n:(0,\infty)\longrightarrow(0,\infty)$ is a decreasing (on $(u_0,\infty)$ for some $u_0\geq0$) and regularly varying function with index $\delta\leq-2$ (that is, $\lim_{u\longrightarrow\infty}n(\lambda u)/n(u)=\lambda^{\delta}$ for all $\lambda>0$). Observe that, due to \cite[Theorem 1.5.11]{goldie}, for any $-4\leq\delta\leq-2$, $$\lim_{\rho\longrightarrow\infty}\frac{\nu(B^{c}_\rho(0))}{\rho^{2}n(\rho)}=\frac{1}{2\pi(-2-\delta)}\quad\textrm{and}\quad \lim_{\rho\longrightarrow\infty}\frac{\int_{B_\rho(0)}|y|^{2}\nu(dy)}{\rho^{4}n(\rho)}=\frac{1}{2\pi(4+\delta)}.$$ Consequently,  Proposition \ref{p5} and \cite[Proposition 1.3.6]{goldie} yield that
\begin{itemize}
  \item[(i)]  if $\delta<-4$, then $\process{L}$ is recurrent.
  \item [(ii)] if $-4<\delta\leq-2$, then $\process{L}$ is transient.
  \item [(iii)] if $\delta=-4$, then $\process{L}$ is transient if $$\int_r^{\infty}\frac{d\rho}{\rho^{5}n(\rho)}<\infty\quad\textrm{for some}\ r>0.$$
\end{itemize}

}\end{example}

\section{Comparison of L\'evy and L\'evy-Type Processes}\label{s5}
In this section, we provide some comparison conditions for the recurrence and transience in terms of the L\'evy measures. Again, we assume that the symbol $q(x,\xi)$ of a two-dimensional L\'evy-type process $\process{F}$ is radial in the co-variable.
\begin{theorem}  \label{tm3.10}
Let $\process{F}$ and  $\process{\bar{F}}$ be  two-dimensional
L\'evy-type processes with  symbols $q(x,\xi)$ and
$\bar{q}(x,\xi)$ and  L\'evy measures $\nu(x,dy)$ and
$\bar{\nu}(x,dy)$, respectively. Assume  that
\begin{enumerate}
\item [(i)]$\nu(x,du\times\R)$ is quasi-unimodal uniformly in $x\in\R^{2}$;
\item [(ii)] there exists $u_0\geq0$ such that  $\nu\left(x,B^{c}_u(0)\right)\geq\bar{\nu}\left(x,B^{c}_u(0)\right)$ (or $\nu\left(x,(u,\infty)\times\R\right)\geq\bar{\nu}\left(x,(u,\infty)\times\R\right)$) for all $x\in\R^{2}$ and all $u\geq u_0$.
 \end{enumerate}   Then,
$$\int_{B_r(0)}\frac{d\xi}{\sup_{x\in\R^{2}}q(x,\xi)}=\infty\quad \textrm{for all}\ r>0$$
implies
$$\int_{B_r(0)}\frac{d\xi}{\sup_{x\in\R^{2}}\bar{q}(x,\xi)}=\infty\quad \textrm{for all}\ r>0.$$
In addition, if $q(x,\xi)$ satisfies \eqref{eq3.8}, then
$$\int_{B_r(0)}\frac{d\xi}{\inf_{x\in\R}\bar{q}(x,\xi)}<\infty\quad \textrm{for some}\ r>0$$
implies
$$\int_{B_r(0)}\frac{d\xi}{\inf_{x\in\R}q(x,\xi)}<\infty\quad \textrm{for some}\ r>0.$$
\end{theorem}
\begin{proof}
The assertion of the theorem is a direct consequence of Theorems \ref{tm3.7} and \ref{tm4}.
\end{proof}

\begin{corollary}\label{c3.11}
 Let $\process{F}$  be a
two-dimensional L\'evy-type process with  symbol
$q(x,\xi)$ and  L\'evy measure $\nu(x,dy)$. Assume that
there exists $x_0\in\R^{2}$ such that
\begin{enumerate}
\item [(i)] $\sup_{x\in\R^{2}}q(x,\xi)=q(x_0,\xi)$ for all $|\xi|$ small enough;
              \item [(ii)] there exists a
two-dimensional  L\'evy process  $\process{L}$ with
symbol $q(\xi)$ and  L\'evy measure $\nu(dy)$, such that $q(\xi)$ is radial,
$\nu(du\times\R)$ is quasi-unimodal and    $\nu\left(x,B^{c}_u(0)\right)\geq\nu\left(x_0,B^{c}_u(0)\right)$ (or $\nu((u,\infty)\times\R)\geq\nu(x_0,(u,\infty)\times\R)$) for  all $u\geq u_0$, for  some $u_0\geq0$.
 \end{enumerate}
 Then, the recurrence property of $\process{L}$ implies
\eqref{eq1.1}.
 Further, if  there exists $x_0\in\R^{2}$ such that
\begin{enumerate}
\item [(i)] $\inf_{x\in\R^{2}}q(x,\xi)=q(x_0,\xi)$ for all $|\xi|$ small enough;
              \item [(ii)] the  measure $\nu(x_0,du\times\R)$ is
              quasi-unimodal and $\int_{\R^{2}}|y|^{2}\nu(x_0,dy)=\infty$;
              \item [(iii)] there exists a
two-dimensional L\'evy process  $\process{L}$ with
symbol $q(\xi)$ and  L\'evy measure $\nu(dy)$, such that $q(\xi)$ is radial and
 $\nu\left(x_0,B^{c}_u(0)\right)\geq\nu\left(x,B^{c}_u(0)\right)$ (or $\nu(x_0,(u,\infty)\times\R)\geq\nu((u,\infty)\times\R)$) for   all $u\geq u_0$, for  some $u_0\geq0$,
 \end{enumerate}
then the transience property of $\process{L}$ implies
\eqref{eq1.2}.
\end{corollary}
\begin{proof} The claims of the corollary follow directly from \cite[Corollary
37.6]{sato-book} and Theorem \ref{tm3.10}.
\end{proof}

Examples of L\'evy-type processes which
satisfy the conditions in  Corollary
\ref{c3.11}
can be found in the class of
 Feller processes obtained by variable order subordination.
More precisely, let
$q(\xi)$ be the symbol  of some $d$-dimensional symmetric L\'evy process (that is, $q(\xi)= q(-\xi)$ for all $\xi\in\R^{d}$), and let $f :\R^{d}\times[0,
\infty)\longrightarrow[0,\infty)$ be a measurable function such that
$\sup_{x\in\R^{d}} f(x, t)\leq c(1 + t)$ for some $c\geq0$ and all
$t\in[0,\infty)$ and for fixed $x\in\R^{d}$ the function
$t\longrightarrow f(x, t)$ is a Bernstein function with $f(x, 0) =
0$. Bernstein functions are the characteristic Laplace exponents of
subordinators (L\'evy processes with nondecreasing sample paths).
For more on Bernstein functions we refer the readers to the
monograph   \cite{Bernstein}. Now, since $q(\xi)\geq0$ for all
$\xi\in\R^{d}$, the function
$$\bar{q}(x,\xi) := f(x,q(\xi)),\quad x,\xi\in\R^{d},$$ is well
defined  and, according to \cite[Theorem 5.2]{Bernstein} and
\cite[Theorem 30.1]{sato-book},  $\xi\longmapsto \bar{q}(x,\xi)$  is
a continuous and negative definite function  satisfying conditions
(\textbf{C2}) and (\textbf{C3}). Hence, $\bar{q}(x,\xi)$ is possibly
the symbol of some  L\'evy-type process. Of special interest is
the case when $f(x,t)=t^{\alpha(x)},$ where
$\alpha:\R^{d}\longmapsto [0,1]$, that is,  $\bar{q}(x,\xi)$
describes variable order subordination. For sufficient conditions on
the symbol $q(\xi)$ and function $\alpha(x)$ such that
$\bar{q}(x,\xi)$ is the symbol of some L\'evy-type process see
\cite{jacob} and \cite{hoh} and  the references therein. Now, assume that
$\alpha(\underline{x})=\inf_{x\in\R^{d}}\alpha(x)\leq\sup_{x\in\R^{d}}\alpha(x)=\alpha(\overline{x})$, for some $\overline{x},\underline{x}\in\R^{d}$.
Then, since the symbol $q(\xi)$ is continuous and $q(0)=0$, there
exists $r>0$ small enough such that $q(\xi)\leq1$ for all
$|\xi|<r.$ In particular, $$q^{\alpha(\overline{x})}(\xi)=\inf_{x\in\R^{d}}q^{\alpha(x)}(\xi)=\inf_{x\in\R^{d}}\bar{q}(x,\xi)\leq\sup_{x\in\R^{d}}\bar{q}(x,\xi)=\sup_{x\in\R^{d}}q^{\alpha(x)}(\xi)=q^{\alpha(\underline{x})}(\xi),\quad |\xi|<r.$$
Let us remark that
 when $q(\xi)$ is the symbol of a standard Brownian motion, then by
variable order subordination we get a  stable-like process.

 \section*{Acknowledgement} This work has been supported in part by the Croatian Science Foundation under Project 3526, NEWFELPRO Programme under Project 31 and  Dresden Fellowship Programme. The author
would like to thank the anonymous reviewer for careful reading of the paper
and for helpful comments that led to improvement of the
presentation.

\bibliographystyle{alpha}
\bibliography{References}

\end{document}